\documentclass[10pt,a4paper]{amsart}
\usepackage[latin1]{inputenc}
\usepackage{amsmath}
\usepackage{amsfonts}
\usepackage{amssymb}
\usepackage{color}

\newtheorem{theorem}{Theorem}
\newtheorem{lemma}[theorem]{Lemma}

\theoremstyle{definition}

\newtheorem{proposition}[theorem]{Proposition}
\newtheorem{corollary}[theorem]{Corollary}

\theoremstyle{remark}
\newtheorem{remark}[theorem]{Remark}

\numberwithin{equation}{section}

\begin{document}

\title[Gradient Ricci solitons with constant scalar curvature]{On gradient Ricci solitons with constant scalar curvature}

\author{Manuel Fern\'{a}ndez-L\'{o}pez and Eduardo Garc\'{\i}a-R\'{\i}o}
\thanks{Supported by projects GRC2013-045 and MTM2013-41335-P with FEDER funds (Spain)}
\address{IES Mar\'ia Sarmiento, Conseller\'ia de educaci\'on, Xunta de Galicia, Spain} \email{manufl@edu.xunta.es}
\address{Faculty of Mathematics, University of Santiago de Compostela, Spain} \email{eduardo.garcia.rio@usc.es}

\subjclass[2010]{53C25, 53C20, 53C44}

\keywords{Gradient Ricci soliton, scalar curvature.}

\begin{abstract}
We use the theory of isoparametric functions to investigate gradient Ricci solitons with constant scalar curvature. We show rigidity of gradient Ricci solitons with constant scalar curvature under some conditions on the Ricci tensor, which are all satisfied if the manifold is curvature homogeneous. This leads to a complete description of four- and six-dimensional K\"ahler gradient Ricci solitons with constant scalar curvature.
\end{abstract}

\maketitle

\section{Introduction and main results}

A complete $n$-dimensional Riemannian manifold $(M,g)$ is said to be a \emph{gradient Ricci soliton} if there exists a smooth function $f$ on $M$ such that
\begin{equation}\label{gradientsoliton}
Rc+H_f=\lambda g,
\end{equation}
where $Rc$ is the Ricci tensor, $H_f$ denotes the Hessian of the function $f$, and $\lambda$ is a real number. For $\lambda >0$ the Ricci soliton is \emph{shrinking}, for $\lambda=0$ it is \emph{steady} and for $\lambda <0$ it is \emph{expanding}. The function $f$ is called the \emph{potential function} of the gradient Ricci soliton. Gradient Ricci solitons play an important role in Hamilton's Ricci flow as they correspond to self-similar solutions, and often arise as singularity models. Therefore it is important to classify gradient Ricci solitons and to understand their geometry. We refer to  \cite{C, Chow, CLN} and the references therein for background on Ricci solitons.

The Ricci soliton equation \eqref{gradientsoliton} links geometric information about the curvature of the manifold through the Ricci tensor and the geometry of the level sets of the potential function by means of their second fundamental form. Hence, classifying gradient Ricci solitons under some curvature conditions is a natural problem. Gradient Ricci solitons with constant scalar curvature were investigated by Petersen and Wylie in \cite{PW1}, who showed that constant scalar curvature is a very restrictive condition in the steady case, since it leads to Ricci flatness. Therefore we focus on the non-steady case in what follows.

If a non-steady gradient Ricci soliton has constant scalar curvature $R$, then it is bounded as $0\leq R\leq n\lambda$ in the shrinking case, and  $n\lambda\leq R\leq 0$ in the expanding case. Our first result shows that the possible values of the scalar curvature are quantified by the soliton constant $\lambda$ as follows

\begin{theorem}\label{csc}
Let $(M^n,g)$ be an $n$-dimensional complete gradient Ricci soliton with constant scalar curvature $R.$ Then $R\in\{0,\lambda,\cdots , (n-1)\lambda, n\lambda\}.$
\end{theorem}

Petersen and Wylie  showed that the extremal values for the scalar curvature in Theorem \ref{csc} are achieved if and only if the underlying Riemannian structure is Einstein \cite{PW1}.
We investigate the realizability of the extremal values $R=\lambda$ and $R=(n-1)\lambda$, showing that the value  $R=\lambda$ cannot occur in the shrinking case, while any complete gradient Ricci soliton with constant scalar curvature $R=(n-1)\lambda$ is necessarily rigid (see Theorem \ref{rigid3l}).

A gradient soliton is said to be \emph{rigid} if it is isometric to a quotient of $N\times \mathbb{R}^k$ where $N$ is an Einstein manifold and $f=\frac{\lambda}{2}|x|^2$ on the Euclidean factor. That is, the Riemannian manifold $(M,g)$ is isometric to $N\times_{\Gamma}\mathbb{R}^k$, where $\Gamma$ acts freely on $N$ and by orthogonal transformations on $\mathbb{R}^k$. Clearly any rigid gradient Ricci soliton has constant scalar curvature. When $M$ is compact, a Ricci soliton $(M,g)$ is rigid if and only if it is Einstein, and the constancy of the scalar curvature characterizes rigidity as a consequence of the Hopf maximum principle (see, for example \cite{ELNM}). In the more general setting of complete non-compact gradient Ricci solitons, Petersen and Wylie showed rigidity under some additional assumptions \cite{PW1}.
One of such conditions is the harmonicity of the Riemann curvature tensor. In the shrinking case, a gradient Ricci soliton is rigid if and only if the Weyl tensor is harmonic \cite{FLGR, MS}.

Our main result shows rigidity of gradient Ricci solitons with constant scalar curvature under the assumption that the Ricci operator has constant rank, from where it follows that curvature homogeneous gradient Ricci solitons are rigid.

\begin{theorem}\label{rigidity}
A complete gradient Ricci soliton is rigid if and only if the Ricci operator has constant rank.
\end{theorem}

\begin{remark}\label{re-6}
A special family of rigid gradient Ricci solitons are the homogeneous ones. Indeed, it was shown in  \cite{PW2} that any homogeneous gradient Ricci soliton is rigid. An immediate application of Theorem \ref{rigidity} shows that the homogeneity condition can be replaced by the more general one of curvature homogeneity. A Riemannian manifold $(M,g)$ is said to be $k$-\emph{curvature homogeneous} if for any two points $p,q\in M$ there exists a linear isometry $\varphi_{pq}:T_pM\rightarrow T_qM$ which preserves the curvature tensor and its covariant derivatives up to order $k$, i.e., $\varphi_{pq}^*\nabla^lR(q)=\nabla^lR(p)$  for all $l=0,\dots,k$. Clearly, any locally homogeneous Riemannian manifold is $k$-curvature homogeneous for all $k$, and conversely, if $(M^n,g)$ is $k$-curvature homogeneous for sufficiently large $k$ (for instance $k\geq \frac{1}{2}n(n-1)$), then it is locally homogeneous. However, there are plenty of examples of Riemannian manifolds which are $0$-curvature homogeneous but not homogeneous (even there are many curvature homogenous Riemannian manifolds whose curvature tensor does not correspond to any homogeneous space~\cite{BoKoVa}). 
Since any curvature homogeneous Riemannian manifold has constant Ricci curvatures, the Ricci operator has constant rank, and thus it follows from Theorem \ref{rigidity} that any $0$-curvature homogeneous complete gradient Ricci soliton is rigid, which generalizes \cite[Theorem 1.1]{PW2}.
\end{remark}

Gradient Ricci solitons with constant scalar curvature are rigid in dimension three \cite{PW3} (in the shrinking case rigidity was shown without any additional assumption in \cite{CCZ}). We extend the above results to the case of K\"ahler Ricci solitons by showing that

\begin{theorem}
\label{kahler-case}
Any K\"ahler gradient Ricci soliton with constant scalar curvature is rigid in dimension $n=4,6$.
\end{theorem}

Previous result strongly depends on the number of different Ricci curvatures and does not hold necessarily in the general non-K\"ahler case, where we have:

\begin{theorem}\label{thm-3}
Any four-dimensional complete gradient shrinking Ricci soliton with constant scalar curvature $R\neq 2\lambda$  is rigid. Moreover, any four-dimensional gradient shrinking (resp., expanding) Ricci soliton with constant scalar curvature $R=2\lambda$ has non-negative (resp., non-positive) Ricci curvature.
\end{theorem}

\section{Proof of the results.}

The results will be obtained by considering the geometric information underlying equation \eqref{gradientsoliton}. The analysis of the level sets of the potential function plays a crucial role since our assumption of constant scalar curvature imposes some restrictions on their geometry. First of all we introduce some definitions to be used in what follows.

A non-constant $C^2$ function $f:M\rightarrow \mathbb{R}$ is said to be a {\it transnormal function} if
\begin{equation}\label{trans}
|\nabla f|^2=b(f)
\end{equation}
for some $C^2$ function $b$ on the range of $f$ in $\mathbb{R}.$ The function $f$ is said to be an {\it isoparametric function} if moreover satisfies
\begin{equation}\label{iso}
\Delta f=a(f)
\end{equation} 
for some continous function $a$ on the range of $f$ in $\mathbb{R}$ \cite{B, GT2, Mi, W}.

Equation (\ref{trans}) implies that the level set hypersurfaces of $f$ are parallel hypersurfaces and it follows from equation (\ref{iso}) that these hypersurfaces have constant mean curvature (i.e., they are isoparametric hypersurfaces).

Given a function $f:M\rightarrow \mathbb{R}$ we denote $f_{min}=\min\{f(x)\, /\, x\in M\}$ and $f_{max}=\max\{ f(x)\, /\, x\in M\},$ if they exist. If the function is transnormal we define the sets $M_-=\{x\in M\, /\, f(x)=f_{min}\}$ and $M_+=\{x\in M\, /\, f(x)=f_{max}\},$ which are called the {\it focal varieties} of $f.$ Note that $M_-$ and/or $M_+$ could be the empty set.

If $(M,g)$ is a gradient Ricci soliton the potential function (after a possible rescaling) satisfies $|\nabla f|^2=2\lambda f-R$ (see, for example \cite{Chow, PW1}) and $\Delta f=n\lambda -R$.  Thus, if the Ricci soliton has constant scalar curvature the potential function satisfies \eqref{trans} and \eqref{iso}, thus being an isoparametric function on $(M,g)$. First of all we show that at least one of the focal varieties is necessarily non-empty.

\begin{lemma}\label{le-1}
Let $(M,g)$ be a gradient Ricci soliton with constant scalar curvature. If it is shrinking then $M_-\neq \emptyset$ and if it is expanding then $M_+\neq \emptyset.$
\end{lemma}

\begin{proof}
We need to show that the potential function has a minimum (resp., maximum) if the Ricci soliton is shrinking (resp., expanding). After translating $f$ by a constant we can suppose that $2\lambda f=|\nabla f|^2$.  This shows that $f$ has the same sign as $\lambda$, and moreover it vanishes at the same points where $\nabla f$ does so. Hence, if $p\in M$ is such that $\nabla f(p)=0$, then $f(p)=0$ and the potential function attains a minimum (resp., maximum) at $p\in M$.

We will argue by contradiction, assuming that $\nabla f\neq 0$ at any point to get a contradiction. Define the function $r:(M,g)\rightarrow\mathbb{R}$ by $r=\sqrt{\frac{2 f}{\lambda}}$. Now it follows that $\nabla r= \frac{\sqrt{\lambda}\nabla f}{2|\nabla f|}$, from where we have that $r$ satisfies the Eikonal equation $|\nabla r|^2=1.$ Thus $r$ is a distance function and the integral curves of $\nabla r$ are geodesics. Hence $\nabla r$ is a complete vector field, due to completeness of $(M,g)$. So, if $\nabla f\neq 0$ then the range of $r$ must be $\mathbb{R},$ which contradicts the fact that $r$ is non-negative. This finishes the proof.
\end{proof}

\begin{remark}\rm
Cao and Zhou \cite{CZ} showed that if $(M,g)$ is a shrinking gradient Ricci soliton, then the potential function satisfies
$\displaystyle
\frac{\lambda}{2} (r(p)-c_1)^2 \leq f(p) \leq \frac{\lambda}{2} (r(p)-c_1)^2$, 
where $r(x)=d(p_0,p)$ is the distance function from some fixed point $p_0 \in M$, and $c_1$ and $c_2$ are positive constants. Previous inequalities provide an alternative proof of the fact that  $f$ attains a global minimum in the shrinking case.
\end{remark}

\begin{proof}[Proof of Theorem \ref{csc}.]
Wang proved in \cite{W} that the focal varieties $M_-$ and $M_+$ of a transnormal function are smooth submanifolds of $M$ (whenever they are not empty). 
Also it was shown in \cite{W} that the restriction of the Hessian $H_f$ of a transnormal function $f$ satisfying \eqref{trans} to $M_{\pm}$ has only two eigenvalues, $0$ and $\frac12 b'(f)$. Indeed, one has $H_f(X,Y)=0$ for all $X,Y\in TM_\pm$, and $H_f(V,W)=\frac12 b'(f) g(V,W)$ for all $V,W\in TM^ {\perp}_\pm.$

In our case, since $b(f)=2\lambda f-R$, we have that $b'(f)=2\lambda$. Now, it follows from (\ref{gradientsoliton}) that the restriction of the Ricci tensor to the (non-empty) focal submanifolds $M_\pm$ is of the form
$$
Rc_{|_{M_\pm}}=
\left(
\begin{array}{cc} 
\lambda I_k & 0 \\ 
0 & 0_{n-k} 
\end{array}\right),
$$ 
where $k=\operatorname{rank}\, Rc_{|_{M_\pm}}$. 
Previous lemma shows that at least one of the focal varieties is not empty, and hence taking traces one obtains that the possible values of the scalar curvature are quantified by the Ricci soliton constant $\lambda$, thus proving Theorem \ref{csc}.
\end{proof}

A basic fact in our analysis is the following \cite[Proposition 1.3]{PW1}:
a gradient shrinking (resp., expanding) Ricci soliton is rigid if and only if it has constant scalar curvature and $0\leq Rc\leq \lambda$ (resp., $\lambda\leq Rc\leq 0)$.

\begin{proof}[Proof of Theorem \ref{rigidity}.]
Recall that the $f$-Laplacian of the scalar curvature satisfies
$\Delta_fR=\lambda R-|Rc|^2$, from where $|Rc|^2=\lambda R$ in the constant scalar curvature setting
(see, for example \cite{PW1}).
Next, assuming that the rank of the Ricci operator is constant, say $\operatorname{rank}\, Rc=\operatorname{dim}\,M_\pm=k$, one has that the scalar curvature is $R=k\lambda$, since it is constant. Let us denote by $R_i$, $i=1,\dots,k,$ the non-zero eigenvalues of the Ricci operator at any point of $M$. Then
$$
\begin{array}{rcl}
\displaystyle
\sum_{i=1}^{k}\left(R_i-\lambda\right)^2
&=&|Rc|^2-2\lambda R + k\lambda^2\\
\noalign{\medskip}
&=&\lambda R-2\lambda R+\lambda R
=0\,.
\end{array}
$$
Hence all non-zero eigenvalues satisfy $R_1=R_2=\cdots =R_k=\lambda$, and thus it follows that  Ricci tensor has only two eigenvalues, $0$ and $\lambda$, from where it follows that  $0\leq Rc\leq\lambda$ and rigidity  is a consequence of the results in \cite{PW1}.

The proof is analogous in the expanding case, using that $M_+\neq \emptyset$ and $Rc_{|_{M_+}}=\operatorname{diag}\, [\lambda,\dots,\lambda,0,\dots,0]$ to obtain that $\lambda\leq Rc\leq 0$.
\end{proof}

\begin{remark}
\rm\label{nuevo-1}
The proof of Theorem \ref{rigidity} relies on the fact that $|Rc|^2=\lambda R$ in the constant scalar curvature setting, from where it follows that $|Rc|^2$ is constant provided that so is $R$. The converse is also true in the non-expanding setting.
In fact, since $R^2\le n|Rc|^2$, one has that $R$ is bounded if $|Rc|^2$ is constant.
If a steady 
gradient Ricci soliton has constant $|Rc|^2$, then it is Ricci flat since at the maximum/supremum of $R$, $\Delta_f R=-2|Rc|^2\leq 0$ and at the minimum/infimum of $R$,
$\Delta_f R=-2|Rc|^2\geq 0$, from where it follows that $|Rc|^2\equiv 0$ and the manifold is Ricci flat. 
Also, 
if $(M,g)$ is a shrinking gradient Ricci soliton then $R$ is constant provided that so is $|Rc|^2$. 
Indeed, at any point where $R$ attains a  maximum/supremum  $R^*$ one has that
$\Delta_f R=\lambda R^*-2|Rc|^2\leq 0$, which implies $\lambda R^* \leq 2|Rc|^2$.
Analogously, at any point where $R$ attains a  minimum/infimum  $R_*$ one has that
$\Delta_f R=\lambda R_*-2|Rc|^2\geq 0$ and thus $\lambda R_* \geq 2|Rc|^2$. Hence
$R^*\leq R_*$, which shows that $R$ is constant.
\end{remark}


It is shown in \cite{PW2} that $3$-dimensional gradient Ricci solitons with constant scalar curvature are rigid. Note that the Ricci tensor of such a soliton has at most three different eigenvalues. We extend this  result in the following

\begin{theorem}\label{threeeigenvalues}
Let $(M,g)$ be a gradient Ricci soliton with constant scalar curvature. If $Rc$ has at most three different eigenvalues then it is rigid.
\end{theorem}

\begin{proof}
Note that $Rc$ has at least one zero eigenvalue due to the constancy of the scalar curvature. Indeed, the Ricci operator satisfies $Rc(\nabla f)=\frac{1}{2}\nabla R$ at any point $p\in M$ where $(\nabla f)(p)\neq 0$, which shows that $Rc$ has a zero eigenvalue on $M\setminus M_{\pm}$ and, as shown in the proof of Theorem \ref{csc}, it also has a zero eigenvalue on $M_{\pm}$. 

Assume that it has two non-zero (possibly) different eigenvalues $R_1,R_2$ of multiplicities $k_1$ and $k_2.$ Since $R$ and $|Rc|^2=\lambda R$ are constant, $R_1$ and $R_2$ must be constant. Indeed, they solve the system of equations
$$
\left\{\begin{array}{lcr}
k_1R_1+k_2R_2=R \\
\noalign{\medskip}
k_1R_1^2+k_2R_2^2=\lambda R
\end{array}
\right.
$$
which only has two solutions if $R_1\neq R_2$ and one solution if $R_1=R_2$. Because of the continuity of the eigenvalues of the Ricci tensor it follows that they are constant, from where we get rigidity.
\end{proof}

\begin{proof}[Proof of Theorem \ref{kahler-case}]
Since for K\"ahler manifolds $(M,g,J)$ the Ricci tensor is
invariant under the action of the complex structure $J$ (that is, it satisfies
${\rm Rc}(J\cdot, J\cdot)= {\rm Rc}(\cdot, \cdot)$),
each Ricci curvature has even multiplicity.
Since the Ricci tensor has a zero eigenvalue with multiplicity at least $2$ it follows that $Rc$ has at most two different eigenvalues if $n=4$ and three different eigenvalues if $n=6.$ The result now follows from Theorem \ref{threeeigenvalues}.
\end{proof}


Theorem \ref{csc} showed that the possible values of the scalar curvature of any gradient Ricci soliton of constant scalar curvature are $R\in\{0,\lambda,\cdots , (n-1)\lambda, n\lambda\}$. Petersen and Wylie proved in \cite{PW1} that the extremal values may occur only for Einstein metrics. Next we show that the soliton is necessarily rigid if $R=(n-1)\lambda$ and that $R\neq \lambda$ in the shrinking case.

\begin{theorem}\label{rigid3l}
Any complete gradient Ricci soliton with constant scalar curvature $R=(n-1)\lambda$  is rigid. Moreover, no complete gradient shrinking Ricci soliton may exist with $R=\lambda$.
\end{theorem}

\begin{proof}
Let us denote by $R_{i}$, $(i=1,\dots,n-1)$ the non-zero eigenvalues of the Ricci operator at any point of $M$. It holds
$$
\begin{array}{rcl}
\displaystyle
\sum_{i=1}^{n-1}\left(R_{i}-\lambda\right)^2
&=&|Rc|^2-2\lambda R + (n-1)\lambda^2\\
\noalign{\medskip}
&=&\lambda R-2\lambda R+\lambda R=0\,.\end{array}
$$
So all non-zero eigenvalues satisfy $R_{1}=R_{2}=\cdots=R_{n-1}=\lambda$. Thus, $Rc$ has constant rank and the Ricci soliton is rigid.

To show the non-existence of complete shrinking solitons with $R=\lambda$ we argue by contradiction. If $R=\lambda$, then the focal variety $M_-\neq \emptyset$ is a minimal submanifold (see \cite{GT2}). Since it has dimension one it must be totally geodesic. Moreover it is compact since $f$ is proper. So we have that $M_-$ is $\mathbb{S}^1.$ Now, $M$ is diffeomorphic to a vector bundle over $\mathbb{S}^1$ (see \cite{B, Mi}). But this contradicts the fact that $M$ has finite fundamental group \cite{Wy}.
\end{proof}

\begin{remark}
\rm
It is shown in \cite[Corollary 2]{PW3} that the existence of a non-zero eigenvalue of the Ricci operator of multiplicity $n-1$ leads to rigidity. The assumption $R=(n-1)\lambda$ in Theorem \ref{rigid3l} is equivalent to the existence of an eigenvalue of multiplicity $n-1$ along the focal varieties $M_\pm$.
\end{remark}

Although our results are not so conclusive for other values of the scalar curvature, the Ricci curvature has sign as follows

\begin{proposition}\label{prop}
Let $(M,g)$ be a gradient shrinking (resp., expanding) Ricci soliton with constant scalar curvature $R=k\lambda$. If  $\operatorname{rank} Rc \leq k+1$, then the Ricci curvature is non-negative (resp., non-positive).
\end{proposition}

\begin{proof}
Let us denote by $R_{i}$, $(i=1,\dots,k+1)$ the possibly non-zero eigenvalues of the Ricci operator at any point of $M$. We have that
$$
\sum_{i=1}^{k+1} R_{i}^2=|Rc|^2=\lambda R=k\lambda^2 = \frac{1}{k} \left(\sum_{i=1}^{r+1} R_{i}\right)^2.
$$
Now, adapting the proof of \cite[Lemma 1]{Ch}, one has that for any reals
$\alpha_i$, $i=1,\dots,r$  ($r\geq 2$) such that
$$
\displaystyle \sum_{i=1}^r \alpha_i^2\leq \frac{1}{r-1} \left(\sum_{i=1}^r \alpha_i\right)^2,
$$
then all the $\alpha_i$'s are non-negative or non-positive.

Hence all non-zero eigenvalues of the Ricci operator are positive or negative. In the shrinking case it is $R>0$ from where we get $Rc\geq 0$, and in the expanding case it holds $R<0$, which gives $Rc\leq 0$.
\end{proof}

\begin{remark}\label{re:12}
\rm
If the scalar curvature of a gradient Ricci soliton satisfies $R=(n-2)\lambda$, then $\operatorname{rank} {Rc} \leq n-1$ since $Rc$ has at least a zero eigenvalue due to the constancy of the scalar curvature. Hence the Ricci curvature is non-negative (resp., non-positive) in the shrinking (resp., expanding) case.
\end{remark}

\begin{proof}[Proof of Theorem \ref{thm-3}]
It immediately follows from Theorem \ref{rigid3l} and Proposition \ref{prop}.
\end{proof}

Recall that a gradient shrinking (expanding) Ricci soliton is rigid if and only if it has constant scalar curvature and $0\leq Rc\leq \lambda$ (resp., $\lambda\leq Rc\leq 0)$ \cite{PW1}. 
Next, we compute the $f$-Laplacian of the Ricci tensor acting on the gradient of the potential function in terms of the eigenvalues of the Ricci tensor aimed to show that if the scalar curvature is constant and $(\Delta_f Rc)(\nabla f,\nabla f)=0$, then rigidity of gradient shrinking (resp., expanding) Ricci solitons is characterized by only one of the bounds: $Rc\geq 0$ or $Rc\leq \lambda$ (resp., $Rc\leq 0$ or $Rc\geq \lambda$).

\begin{lemma}\label{deltaf}
Let $(M,g)$ be a gradient Ricci soliton with constant scalar curvature. If $R_{i}$ denote the eigenvalues of the Ricci tensor, then it holds
$$
(\Delta_fRc)(\nabla f,\nabla f)=2\sum_{i=1}^n (\lambda -R_{i})^ 2 R_{i}.
$$
\end{lemma}

\begin{proof}
Let us consider $p\in M$ such that $(\nabla f)(p)\neq 0$. Let $\{E_1,\cdots, E_n\}$ be eigenvectors of $Rc$ at $p$ and take normal coordinates centered at $p$ induced by the orthonormal basis $\{ E_1,\dots, E_n\}$. Then, using that $Rc(\cdot,\nabla f)=0,$ one has
$$
\begin{array}{rcl}
\displaystyle (\Delta_f Rc)(\nabla f,\nabla f) & =& \displaystyle \sum_{i=1}^n (\nabla^2_{E_i,E_i} Rc)(\nabla f,\nabla f) -(\nabla_{\nabla f} Rc) (\nabla f,\nabla f)
\\
& = & \displaystyle
\sum_{i=1}^n (\nabla_{E_i} (\nabla_{E_i} Rc))(\nabla f,\nabla f)
\\
& = & \displaystyle
\sum_{i=1}^n \left[ \nabla_{E_i} ((\nabla_{E_i} Rc))(\nabla f,\nabla f))-2(\nabla_{E_i} Rc))(\nabla_{E_i}\nabla f,\nabla f)\right]
\\
& = & \displaystyle
\sum_{i=1}^n \left(\nabla_{E_i}(-2Rc(\nabla_{E_i}\nabla f,\nabla f))
\right.-2(\nabla_{E_i}Rc)(\nabla_{E_i}\nabla f,\nabla f)
\\
& = & \displaystyle
-2\sum_{i=1}^n [E_i Rc(\nabla_{E_i}\nabla f,\nabla f)-Rc(\nabla_{E_i}\nabla_{E_i}\nabla f,\nabla f)
\\
& & \qquad\quad
- Rc(\nabla_{E_i}\nabla f,\nabla_{E_i}\nabla f)]
\\
& = & \displaystyle
2 \sum_{i=1}^n(\lambda-R_{i})^2 R_{i}.\end{array}$$
\end{proof}

\begin{corollary}
Let $(M,g)$ be a gradient shrinking (resp., expanding) Ricci soliton with constant scalar curvature. If $(\Delta_f Rc)(\nabla f,\nabla f)=0$ and $Rc\geq 0$ or $Rc\leq \lambda$ (resp., $Rc\leq 0$ or $Rc\geq \lambda$) then the soliton is rigid.
\end{corollary}

\begin{proof}
First of all note that, since $|Rc|^2=\lambda R,$ one has
$\sum_{i=1}^n(\lambda-R_{i})^2 R_{i}=\sum_{i=1}^n R_{i}^2(R_{i}-\lambda)$.
Now it follows from the constancy of the scalar curvature and $(\Delta_f Rc)(\nabla f,\nabla f)=0$ that $\sum_{i=1}^n(\lambda-R_{i})^2 R_{i}=0$. Hence
$$
\sum_{i=1}^n(\lambda-R_{i})^2 R_{i}=\sum_{i=1}^n R_{i}^2(R_{i}-\lambda)=0.
$$
Thus, under our assumptions on the bounds of the Ricci tensor we get that $R_{i}\in\{0,\lambda\}$, and thus the Ricci curvatures are constant, which shows rigidity.
\end{proof}

As shown in Theorem \ref{thm-3}, four-dimensional complete shrinking gradient Ricci solitons with constant scalar curvature $R\neq 2\lambda$ are rigid. If the scalar curvature satisfies $R=2\lambda$ we need some additional assumptions to get rigidity.
Before stating the next result, we recall some basic facts about scalar curvature invariants. The space of scalar curvature invariants of order two, $I(1,n)$, is one-dimensional and it is generated by the scalar curvature $R$. The space corresponding to four-order invariants $I(2,n)$ is generated by $\{ R, \Delta R, |Rc|^2, |Rm|^2\}$. Some other contractions of the curvature and the Ricci tensors, giving rise to new $(0,2)$-symmetric tensors fields ($\check{R}$, $\check{Rc}$ and $L(Rc)$), will be used in what follows. With respect to an orthonormal basis, set 
$$
\check{R}_{ij}=\sum_{abc} R_{abci}R_{abcj},\quad
\check{Rc}_{ij}=\sum_{a}Rc_{ai}Rc_{aj},\quad
L(Rc)_{ij}=2\sum_{ab}R_{iabj}Rc_{ab}.
$$
Considering the four-dimensional Gauss-Bonnet integrand $|Rm|^2-4|Rc|^2+R^2$, the following curvature identity is shown in \cite{EPS}:
$$
\check{R}-2\check{Rc}-L(Rc)+R Rc=\frac{1}{4}(|Rm|^2-4|Rc|^2+R^2)g.
$$

Next, observe that the f-Laplacian of the principal Ricci curvatures (see, for example \cite{ELNM}) is given by
$$
\Delta_fR_{i}=2\lambda R_{i}-2 R_{ikis}Rc^{ks}=2\lambda R_i + L(Rc)_{ii}.
$$

\begin{theorem}\label{dimension4}
Let $(M,g)$ be a $4$-dimensional complete gradient shrinking Ricci soliton with constant scalar curvature $R=2\lambda.$ Then it is rigid if and only if one of the following assumptions holds:
\begin{itemize}
\item[(i)] $Rc\leq \lambda$,
\item[(ii)] $(\Delta_f Rc)(\nabla f,\nabla f)=0$,
\item[(iii)] 
$\check{R}(\nabla f,\nabla f)\leq \frac{1}{4}( |Rm|^2-4\lambda^2)g(\nabla f,\nabla f)$

\end{itemize}
\end{theorem}

\begin{proof}[Proof of Theorem \ref{dimension4}]
Under our first assumption, $Rc\leq \lambda$, we have that $0\leq Rc\leq \lambda$, from where it follows rigidity \cite{PW1}.
Now, assume that $(\Delta_f Rc)(\nabla f,\nabla f)=0$. It follows from previous lemma that $\sum_{i=1}^4(\lambda-R_{i})^2 R_{i}=0$. Since $Rc\geq 0$ we have that $R_{i}\in\{0,\lambda\}$. Thus $Rc$ has constant eigenvalues and the soliton is rigid.

Let $\{E_1,E_2,E_3,E_4\}$ be an orthonormal basis consisting of eigenvectors of the Ricci tensor with $E_1=\frac{\nabla f}{|\nabla f|}$. Then $\Delta_f R_1=L(Rc)_{11}$.
Since $R=2\lambda$, it follows from Proposition \ref{prop} and Remark \ref{re:12} that 
$\Delta_f R_{1}=\frac{1}{|\nabla f|^2(p)}(\Delta_f Rc)(\nabla f,\nabla f)\geq 0$.

Now, since the scalar curvature is constant $R=2\lambda$, $Rc(E_1,E_1)=\check{Rc}(E_1,E_1)=0$ and 
$|Rm|^2-4|Rc|^2+R^2=|Rm|^2-4\lambda^2$. Hence the curvature identity above becomes
$$
L(Rc)(E_1,E_1)=\check{R}(E_1,E_1)-\frac{1}{4}|Rm|^2+\lambda^2.
$$
Now, it follows from Assumption $(iii)$ that $L(Rc)(E_1,E_1)\leq 0$, and hence $\Delta_f R_{1}=0$ and rigidity is a consequence of Assumption $(ii)$.

Finally observe that conditions $(i)-(iii)$ are satisfied by any rigid four-dimensional gradient Ricci soliton.
\end{proof}

\begin{remark}\rm\label{nuevo-2}
As a final observation, note that a 
\emph{complete gradient Ricci soliton is rigid if and only if the Ricci operator $Rc$ and its powers $Rc^2$, $Rc^3$ and $Rc^4$ have constant traces.}
In fact, proceeding as in the proof of Theorem \ref{rigidity} one has that the function
$$
\sum_iR_i^2(\lambda-R_i)^2=\sum_iR_i^4-2\lambda\sum_i R_i^3+\lambda^2\sum_i R_i^2
$$ 
is constant provided that the traces of the first powers of the Ricci operator are so. Since its value on the focal submanifold is zero we have that  the function $\sum_iR_i^2(\lambda-R_i)^2$ vanishes on $M$. Thus $R_i\in\{0,\lambda\}$, which shows that the Ricci curvatures are constant, from where it follows that the soliton is rigid.

Further note that while the constancy of the scalar curvature implies the constancy of the trace of 
 $Rc^2$, the functions  $\operatorname{tr}\,Rc^3$ and $\operatorname{tr}\,Rc^4$ are not yet completely understood.
\end{remark}

\end{document}